\documentclass{amsart}

\newtheorem{theorem}{Theorem}[section]

\newtheorem{proposition}[theorem]{Proposition}

\theoremstyle{definition}
\newtheorem{definition}[theorem]{Definition}

\theoremstyle{remark}

\numberwithin{equation}{section}

\begin{document}

\title{ On the Completeness of Gradient Ricci Solitons}

\author{Zhu-Hong Zhang}
\address{Department of Mathematics, Sun Yat-Sen University, Guangzhou, P.R.China 510275}
\email{juhoncheung@sina.com}
\thanks{}



\date{July 10, 2008.}


\keywords{completeness, gradient Ricci soliton, gradient
self-similar solution}

\begin{abstract}
A gradient Ricci soliton is a triple $(M,g,f)$ satisfying
$R_{ij}+\nabla_i\nabla_j f=\lambda g_{ij}$ for some real number
$\lambda$. In this paper, we will show that the completeness of the
metric $g$ implies that of the vector field $\nabla f$.
\end{abstract}

\maketitle



\section{Introduction}

\begin{definition}
Let $(M,g,X)$ be a smooth Riemannian manifold with $X$ a smooth
vector field. We call $M$ a Ricci soliton if
$Ric+\frac{1}{2}{\mathcal{L}}_X g={\lambda}g$ for some real number
$\lambda$. It is called shrinking when $\lambda>0$, steady when
$\lambda=0$, and expanding when $\lambda<0$. If $(M,g,f)$ is a
smooth Riemannian manifold where $f$ is a smooth function, such that
$(M,g,\nabla f)$ is a Ricci soliton, i.e. $R_{ij}+\nabla_i\nabla_j
f=\lambda g_{ij}$, we call $(M,g,f)$ a gradient Ricci soliton and
$f$ the soliton function.
\end{definition}

On the other hand, there has the following definition (see chapter 2
of \cite{{CK}}).

\begin{definition}
Let $(M,g(t),X)$ be a smooth Riemannian manifold with a solution
$g(t)$ of the Ricci flow on a time interval $(a,b)$ containing 0,
where $X$ is smooth vector field. We call $(M,g(t),X)$ self-similar
solution if there exist scalars $\sigma(t)$ such that
$g(t)=\sigma(t)\varphi_t^*(g_0)$, where the diffeomorphisms
$\varphi_t$ is generated by $X$. If the vector field $X$ comes from
a gradient of a smooth function $f$, then we call $(M,g(t),f)$ a
gradient self-similar solution.
\end{definition}

It is easy to see that if $(M,g(t),f)$ is a complete gradient
self-similar solution, then $(M,g(0),f)$ must be a complete gradient
Ricci soliton. Conversely, when $(M,g,f)$ is a complete gradient
Ricci soliton and in addition, the vector field $\nabla f$ is
complete, it is well known (see for example Theorem 4.1 of
\cite{CLN}) that there is a complete gradient self-similar solution
$(M,g(t),f)$, $t\in (a,b)$ (with $0\in (a,b)$), such that $g(0)=g$.
Here we say that a vector field $\nabla f$ is complete if it
generates a family of diffeomorphisms $\varphi_t$ of $M$ for $t\in
(a,b)$.

So when the vector field is complete, the definitions of gradient
Ricci soliton and gradient self-similar solution are equivalent. In
literature, people sometimes confuse the gradient Ricci solitons
with the gradient self-similar solutions. Indeed, if the gradient
Ricci soliton has bounded curvature, then it is not hard to see that
the vector field $\nabla f$ is complete. But, in general the soliton
does not have bounded curvature.

The purpose of this paper is to show that the completeness of the
metric $g$ of a gradient Ricci soliton $(M,g,f)$ implies that of the
vector field $\nabla f$, even though the soliton does not have
bounded curvature. Our main result is the following

\begin{theorem}
Let $(M,g,f)$ be a gradient Ricci soliton. Suppose the metric $g$ is
complete, then we have:

(i) $\nabla f$ is complete;

(ii) $R\ge0$, if the soliton is steady or shrinking;

(iii) $\exists C\ge0$, such that $R\ge-C$, if the soliton is
expanding.
\end{theorem}

Indeed, we will show that the vector field $\nabla f$ grows at most
linearly and so it is integrable. Hence the above Definition 1.1 and
1.2 are equivalent when the metric is complete.

\vskip 0.3cm \noindent {\bf Acknowledgement}  I would like to thank
my advisor Professor X.P.Zhu for many helpful suggestions and
discussions.

\section{Gradient Ricci Solitons}

Let $(M,g,f)$ be a gradient Ricci soliton, i.e.,
$R_{ij}+\nabla_i\nabla_j f=\lambda g_{ij}$. By using the contracted
second Bianchi identity we get the equation
$R+|{\nabla}f|^2-2{\lambda}f=const$.

\begin{definition}
Let $(M,g,f)$ be a gradient shrinking or expanding soliton. By
rescaling $g$ and changing $f$ by a constant we can assume
$\lambda\in \{-\frac{1}{2},\frac{1}{2}\}$ and
$R+|{\nabla}f|^2-2{\lambda}f=0$. We call such a soliton normalized,
and $f$ a normalized soliton function.
\end{definition}

\begin{proposition}
Let $(M,g,f)$ be a gradient Ricci soliton. Fix $p$ on $M$, and
define $d(x)\stackrel{\mathrm{\Delta}}{=}d(p,x)$, then the following
hold

(i) $\triangle R=<\nabla f,\nabla R>+2\lambda R-|Ric|^2$;

(ii) Suppose $Ric\le (n-1)K$ on $B_{r_0}(p)$, for some positive
numbers $r_0$ and $K$. Then for arbitrary point $x$, outside
$B_{r_0}(p)$, we have
$$\triangle d-<\nabla f,\nabla d>\le -\lambda
d(x)+(n-1)\Big\{\frac{2}{3}Kr_0+r_0^{-1}\Big\}+|\nabla f|(p).$$
\end{proposition}

\begin{proof}
(i) By using the soliton equation and the contracted second Bianchi
identity $\nabla_i R=2g^{jk}\nabla_j R_{ik} $, we have
$$\arraycolsep=1.5pt\begin{array}{rcl}
\triangle R &=& g^{ij}\nabla_i\nabla_j R = g^{ij}\nabla_i(2g^{kl}R_{jk}\nabla_l f)=2g^{ij}g^{kl}\nabla_i(R_{jk}\nabla_l f)\\[4mm]
    &=& 2g^{ij}g^{kl}\nabla_i(R_{jk})\nabla_l f+2g^{ij}g^{kl}R_{jk}\nabla_i\nabla_l f\\[4mm]
    &=& g^{kl}\nabla_k R\nabla_l f+2g^{ij}g^{kl}R_{jk}(\lambda g_{il}-R_{il})\\[4mm]
    &=& <\nabla f,\nabla R>+2\lambda R-2|Ric|^2.\\[4mm]
\end{array}$$

(ii) Let $\gamma:[0,d(x)]\rightarrow M$ be a shortest normal
geodesic from $p$ to $x$. We may assume that $x$ and $p$ are not
conjugate to each other, otherwise we can understand the
differential inequality in the barrier sense. Let
$\{\dot{\gamma}(0),e_1,\cdots,e_{n-1}\}$ be an orthonormal basis of
$T_pM$. Extend this basis parallel along $\gamma$ to form a parallel
orthonormal basis $\{\dot{\gamma}(t),e_1(t),\cdots,e_{n-1}(t)\}$
along $\gamma$.

Let $X_i(t)$, $i=1,2,\cdots,n-1$, be the Jacobian fields along
$\gamma$ with $X_i(0)=0$ and $X_i(d(x))=e_i(d(x))$. Then it is
well-known that (see for example \cite{SY})
$$\triangle d(x)=\sum\limits_{i=1}^{n-1}\int_0^{d(x)}[|\dot{X}_i|^2-R(\dot{\gamma},X_i,\dot{\gamma},X_i)]dt.$$

Define vector fields $Y_i$, $i=1,2,\cdots,n-1$, along $\gamma$ as
follows
$$
Y_i(t)= \left\{
       \begin{array}{lll}
       \frac{t}{r_0}e_i(t),\quad  &if\ t\in [0,r_0];\\[4mm]
       e_i(t),\quad  &if\ t\in [r_0,d(x)].
       \end{array}
    \right.
$$
Then by using the standard index comparison theorem we have
$$\arraycolsep=1.5pt\begin{array}{rcl}
\triangle d(x) &=& \sum\limits_{i=1}^{n-1}\int_0^{d(x)}[|\dot{X}_i|^2-R(\dot{\gamma},X_i,\dot{\gamma},X_i)]dt\\[4mm]
               &\le& \sum\limits_{i=1}^{n-1}\int_0^{d(x)}[|\dot{Y}_i|^2-R(\dot{\gamma},Y_i,\dot{\gamma},Y_i)]dt\\[4mm]
               &=& \int_0^{r_0}[\frac{n-1}{r_0^2}-\frac{t^2}{r_0^2}Ric(\dot{\gamma},\dot{\gamma})]dt
               +\int_{r_0}^{d(x)}[-Ric(\dot{\gamma},\dot{\gamma}]dt\\[4mm]
               &=& -\int_{0}^{d(x)}Ric(\dot{\gamma},\dot{\gamma})dt
               +\int_0^{r_0}[\frac{n-1}{r_0^2}+(1-\frac{t^2}{r_0^2})Ric(\dot{\gamma},\dot{\gamma})]dt\\[4mm]
               &\le& -\int_{\gamma}Ric(\dot{\gamma},\dot{\gamma})dt+(n-1)\Big\{\frac{2}{3}Kr_0+r_0^{-1}\Big\}.\\[4mm]
\end{array}$$

On the other hand,
$$<\nabla f,\nabla d>(x)=\nabla_{\dot{\gamma}}f(x)=\int_0^{d(x)}(\frac{d}{dt}\nabla_{\dot{\gamma}}f)dt+\nabla_{\dot{\gamma}}f(p)\ge\int_{\gamma}(\nabla_{\dot{\gamma}}\nabla_{\dot{\gamma}}f)dt-|\nabla f|(p).$$

Using the soliton equation, we have
$$\arraycolsep=1.5pt\begin{array}{rcl}
\triangle d-<\nabla f,\nabla d>
    &\le& -\int_{\gamma}\Big[Ric(\dot{\gamma},\dot{\gamma})+\nabla_{\dot{\gamma}}\nabla_{\dot{\gamma}}f\Big]dt+(n-1)\Big\{\frac{2}{3}Kr_0+r_0^{-1}\Big\}+|\nabla f|(p)\\[4mm]
    &=& -\lambda d(x)+(n-1)\Big\{\frac{2}{3}Kr_0+r_0^{-1}\Big\}+|\nabla f|(p).\\[4mm]
    \end{array}$$
\end{proof}

Now we are ready to prove the theorem 1.3 .

\begin{proof}
Fix a point $p$ on $M$, and define
$d(x)\stackrel{\mathrm{\Delta}}{=}d(p,x)$. We divide the argument
into three steps.

\vskip 0.1cm \noindent{\bf Step 1} We want to prove a curvature
estimate in the following assertion.

\vskip 0.1cm \noindent{\bf Claim} \emph{For any gradient Ricci
soliton, we have:}

\emph{(i) If the soliton is shrinking or steady, then $R\ge 0$;}

\emph{(ii) If the soliton is expanding, then there exist a
nonnegative constant $C=C(n)$ such that $R\ge \lambda C$.}

We only prove the case (i), $\lambda\ge0$. Note that there is a
positive constant $r_0$, such that $Ric\le (n-1)r_0^{-2}$ on
$B_{r_0}(p)$, and $|\nabla f|(p)\le (n-1)r_0^{-1}$, then by
Proposition 2.2, we have
$$\triangle d-<\nabla f,\nabla d>\le\frac{8}{3}(n-1)r_0^{-1},$$ for any $x\notin B_{r_0}(p)$.

For any fixed constant $A>2$, we consider the function
$u(x)=\varphi(\frac{d(x)}{Ar_0})R(x)$, where $\varphi$ is a fixed
smooth nonnegative decreasing function such that $\varphi=1$ on
$(-\infty,\frac{1}{2}]$, and $\varphi=0$ on $[1,\infty)$.

Then by Proposition 2.2, we have
$$\arraycolsep=1.5pt\begin{array}{rcl}
\triangle u &=& R\triangle\varphi+\varphi\triangle R+2<\nabla \varphi,\nabla R> \\[4mm]
    &=& R(\varphi''\frac{1}{(Ar_0)^2}+\varphi '\frac{1}{Ar_0}\triangle d)+\varphi(<\nabla f,\nabla R>+2\lambda R-|Ric|^2)+2<\nabla \varphi,\nabla R>. \\[4mm]
\end{array}$$

If $\min\limits_{x\in M}u \ge0$, then $R\ge0$ on
$B_{\frac{1}{2}Ar_0}(p)$. Otherwise, $\min\limits_{x\in M}u <0$,
then there exist some point $x_1\in B_{Ar_0}(p)$, such that
$u(x_1)=\varphi R(x_1)=\min\limits_{x\in M}u<0$. Because $u(x_1)$ is
the minimum of the function $u(x)$, we have $\varphi' R(x_1)>0$,
$\nabla u(x_1)=0,$ and $\triangle u(x_1)\ge0$.

Let us first consider the case that $x_1\notin B_{r_0}(p)$. Then by
direct computation, we have
$$\arraycolsep=1.5pt\begin{array}{rcl}
\triangle u (x_1)
   &=& (\frac{\varphi''}{\varphi}\frac{1}{(Ar_0)^2}+\frac{\varphi '}{\varphi}\frac{1}{Ar_0}\triangle d)u(x_1)-\frac{\varphi'}{\varphi}\frac{1}{Ar_0}<\nabla f,\nabla d>u(x_1) \\[4mm]
   &&\hskip 0.1cm  +2\lambda u(x_1)-\varphi|Ric|^2-\frac{\varphi '^2}{\varphi^2}\frac{2}{(Ar_0)^2}u(x_1) \\[4mm]
   &\le& (\frac{\varphi''}{\varphi}\frac{1}{(Ar_0)^2}-\frac{\varphi '^2}{\varphi^2}\frac{2}{(Ar_0)^2})u(x_1)-\frac{2}{n}\varphi R^2 \\[4mm]
   &&\hskip 0.1cm +\frac{\varphi'}{\varphi}\frac{1}{Ar_0}u(x_1)(\triangle d-<\nabla f,\nabla d>). \\[4mm]
   &\le& (\frac{\varphi''}{\varphi}\frac{1}{(Ar_0)^2}-\frac{\varphi '^2}{\varphi^2}\frac{2}{(Ar_0)^2})u(x_1)-\frac{2}{n}\frac{1}{\varphi}u(x_1)^2 \\[4mm]
   &&\hskip 0.1cm +\frac{8}{3}(n-1)\frac{\varphi'}{\varphi}\frac{1}{Ar_0^2}u(x_1) \\[4mm]
   &=&\frac{u(x_1)}{\varphi}\Big\{(\varphi''\frac{1}{(Ar_0)^2}-\frac{\varphi '^2}{\varphi}\frac{2}{(Ar_0)^2})+\frac{8}{3}(n-1)\varphi'\frac{1}{Ar_0^2}-\frac{2}{n}u(x_1)\Big\}\\[4mm]
   &\le&\frac{|u(x_1)|}{\varphi}\Big\{\frac{\varphi '^2}{\varphi}\frac{2}{Ar_0^2}+\frac{8(n-1)}{3}(-\varphi')\frac{1}{Ar_0^2}+|\varphi''|\frac{1}{Ar_0^2}-\frac{2}{n}|u(x_1)|\Big\}.\\[4mm]
\end{array}$$

Note that there exist a constant
$\widetilde{C}=\widetilde{C}(\varphi)$, such that $|\varphi'|\le
\widetilde{C}$, $\frac{\varphi'^2}{\varphi}\le \widetilde{C}$, and
$|\varphi''|\le \widetilde{C}$. So
$$|u(x_1)|\le \frac{C}{Ar_0^2},$$
where the constant $C=C(\varphi,n)$, i.e., $R\ge-\frac{C}{Ar_0^2}$
on $B_{\frac{1}{2}Ar_0}(p)$.

We now consider the remaining case that $x_1\in B_{r_0}(p)$. Then
$\varphi'(x_1)=\varphi''(x_1)=0$, and we have
$$\triangle u(x_1)=2\lambda u(x_1)-\varphi|Ric|^2\le |u(x_1)|[-2\lambda-\frac{2}{n}|u(x_1)|] .$$

Since $\lambda\ge0$, we have $|u(x_1)|\le 0$, i.e., $u(x_1)=0$. This
is a contradiction.

Combining the above two cases, we have $R\ge-\frac{C}{Ar_0^2}$ on
$B_{\frac{1}{2}Ar_0}(p)$ for any $A>2$, which implies that $R\ge0$
on $M$.

The proof of (ii) is similar.

\vskip 0.1cm \noindent{\bf Step 2} We next want to show that the
gradient field grows at most linearly.

\vskip 0.1cm \noindent{\bf Claim} \emph{For any gradient Ricci
soliton, there exist constants $a$ and $b$ depending only on the
soliton, such that}

(i) $|\nabla f|(x)\le |\lambda|d(x)+a$;

(ii) $|f|(x)\le \frac{|\lambda|}{2}d(x)^2+ad(x)+b$.

For any point $x$ on $M$, we connect $p$ and $x$ by a shortest
normal geodesic $\gamma(t),t\in [0,d(x)]$.

We first consider that the soliton is steady, then $R\ge0$ and
$R+|\nabla f|^2=C\ge0$, so we have $|\nabla f|\le \sqrt{C}$.

Secondly, We consider that the soliton is shrinking. Without loss of
generality, we may assume the soliton is normalized. So $R\ge0$ and
$R+|\nabla f|^2-f=0$, these imply $f\ge |\nabla f|^2$. Let
$h(t)=f(\gamma(t))$, then
$$|h'|(t)= |<\nabla f, \dot{\gamma}>|(t)\le |\nabla
f|(\gamma(t))\le \sqrt{f(\gamma(t))}=\sqrt{h(t)} .$$

By integrating above inequality, we get
$|\sqrt{h(d(x))}-\sqrt{h(0)}|\le \frac{1}{2}d(x)$. Thus $|\nabla
f|(x)\le \frac{1}{2}d(x)+\sqrt{f(p)}$.

Finally, we consider that the soliton is expanding. Similarly we
only need to show the normalized case. So $R\ge-\frac{C}{2}$ and
$R+|\nabla f|^2+f=0$, we obtain $-f+\frac{C}{2}\ge |\nabla f|^2$.
Let $h(t)=-f(\gamma(t))+\frac{C}{2}$, thus $$|h'|(t)= |<\nabla f,
\dot{\gamma}>|(t)\le |\nabla f|(\gamma(t))\le \sqrt{h(t)} .$$

By integrating above inequality, we get
$|\sqrt{h(d(x))}-\sqrt{h(0)}|\le \frac{1}{2}d(x)$. Thus $|\nabla
f|(x)\le \frac{1}{2}d(x)+\sqrt{-f(p)+\frac{C}{2}}$.

Therefore we have proved (i).

The conclusion (ii) follows from (i) immediately.

\vskip 0.1cm \noindent{\bf Step 3} Since the gradient field $\nabla
f$ grows at most linearly, it must be integrable. Thus we have
proved theorem 1.3 .
\end{proof}

\bibliographystyle{amsplain}

\end{document}